\documentclass[numbers]{article}

\usepackage[margin=1in]{geometry}

\usepackage[utf8]{inputenc}
\usepackage{amsmath}
\usepackage{amsthm}
\usepackage{amsfonts}
\usepackage{xcolor}
\usepackage{natbib}
\usepackage{graphicx}
\usepackage{tikz}
\usepackage{caption}
\usepackage{subcaption}
\usepackage{float}

\newtheorem{theorem}{Theorem}
\newtheorem{lemma}[theorem]{Lemma}

\newtheorem{proposition}{Proposition}

\newtheorem{question}{Question}
\newtheorem{claim}{Claim}
\newtheorem{corollary}{Corollary}

\newcommand{\ceil}[1]{\left\lceil #1 \right\rceil}

\usepackage{tikzit}

\tikzstyle{Filled in}=[fill=black, draw=black, shape=circle, scale=0.7]
\tikzstyle{Hollow}=[fill=white, draw=black, shape=circle, tikzit draw=black, tikzit fill=white, scale=0.55]
\tikzstyle{Hollow7}=[fill=white, draw=black, shape=circle, tikzit draw=black, tikzit fill=white, scale=0.7]
\tikzstyle{Forced}=[fill={rgb,255: red,250; green,51; blue,51}, draw=black, shape=circle, scale=0.55]
\tikzstyle{ZFS}=[fill=blue, draw=black, shape=circle, scale=0.55]
\tikzstyle{ZFS_Purple}=[fill={rgb,255: red,126; green,1; blue,157}, draw=black, shape=circle, scale=0.55]
\tikzstyle{ZFS7}=[fill=blue, draw=black, shape=circle, scale=0.7]
\tikzstyle{new style 0}=[fill=lightgray, draw=black, shape=circle, scale=0.4]
\tikzstyle{Trunk}=[fill={rgb,255: red,165; green,42; blue,42}, draw=black, shape=circle, scale=0.7]
\tikzstyle{new style 1}=[fill={rgb,255: red,1; green,138; blue,92}, draw={rgb,255: red,1; green,138; blue,92}, shape=circle, scale]
\tikzstyle{Tree}=[fill={rgb,255: red,1; green,138; blue,92}, draw=black, shape=circle, scale=0.7]
\tikzstyle{Boundary}=[fill={rgb,255: red,222; green,0; blue,222}, draw=black, shape=circle, scale=0.7]

\tikzstyle{Forced Edge}=[-, draw={rgb,255: red,255; green,62; blue,62}, fill={rgb,255: red,255; green,69; blue,69}]
\tikzstyle{Bold}=[-, fill=black]
\tikzstyle{Bold2}=[-, fill=black, line width = 1.3pt]
\tikzstyle{Dual}=[-, fill={rgb,255: red,1; green,138; blue,92}, draw={rgb,255: red,1; green,138; blue,92}]
\tikzstyle{Trunk 2}=[-, fill={rgb,255: red,165; green,42; blue,42}, draw={rgb,255: red,165; green,42; blue,42}]

\title{Zero Forcing on 2-connected Outerplanar Graphs}
\author{Nolan Ison\footnote{Department of Mathematics, Brigham Young University, Provo, UT, \emph{nolan.ison@mathematics.byu.edu}}, Mark Kempton\footnote{Department of Mathematics, Brigham Young University, Provo, UT, \emph{mkempton@mathematics.byu.edu}}, Franklin Kenter\footnote{Department of Mathematics, United States Naval Academy, Annapolis, MD, \emph{kenter@usna.edu}}}
\date{}

\begin{document}

\maketitle


\begin{abstract}
    We determine upper and lower bounds on the zero forcing number of 2-connected outerplanar graphs in terms of the structure of the weak dual. We show that the upper bound is always at most half the number of vertices of the graph. This work generalizes work of Hern\'andez, Ranilla and Ranilla-Cortina in \cite{hernandez2019zero} who proved a similar result for maximal outerplanar graphs.
\end{abstract}



\section{Introduction}

In this paper, all graphs are assumed to be simple, undirected graphs. We will work almost exclusively with 2-connected outerplanar graphs. A graph is \textit{2-connected} if for any $v \in V(G)$, $G - v$ is connected. A graph is \textit{outerplanar} if it is planar and every vertex lies on the outer face of the graph. A way to think of 2-connected outerplanar graphs is by taking a tree and blowing each vertex up into a cycle where neighboring cycles share exactly one edge.

The zero forcing paramater $Z(G)$ was introduced in \cite{work2008zero} to bound the maximum nullity of a simple graph. Let all the vertices in $G$ be either colored blue or white. The \textit{color changing rule} states that if a blue vertex $u$ has exactly one white neighbor $v$, then change the color of $v$ to blue. We say that $u$ forces (or infects) $v$. A \textit{zero forcing set} (ZFS) of $G$ is a set of blue vertices that force the entire graph to be blue. The \textit{zero forcing number} of $G$, denoted as $Z(G)$, is the minimum size of a zero forcing set of $G$. 

 In relation to topological graph theory,  the Colin de Verdi\`ere parameter, $\mu(G)$, is a variant of maximum nullity that perfectly characterizes disjoint paths ($\mu = 1$), outerplanar graphs ($\mu = 2$), planar graphs ($\mu = 3$), and linklessly embeddable graphs ($\mu = 4$) \cite{CDV1}. However, to date, there is little known about zero forcing in the context of topological graph theory.  Butler and Young showed that if $Z(G)=2$ and $G$ is connected, then $G$ is outerplanar \cite{butlerthrot}. For maximal outerplanar graphs (i.e. ``triangluations''), Hern\'andez, Ranilla and Ranilla-Cortina showed that  \begin{equation}\label{eq:maximal} Z(G) \leq 2\ceil{\frac{n_2 + c -1}{2}}\end{equation}
 where $n_2$ is the number of vertices of degree 2 and $c$ is the number of components in a certain subgraph of the weak dual of $G$ \cite{hernandez2019zero}. In our main result, we generalize this result to any 2-connected outerplanar graph, and show that $Z(G) \le \frac{n}{2}$  (Theorem 2) for every such graph. We note that the quantity from equation \ref{eq:maximal} will not apply to all 2-connected outerplanar graphs, as we see in Figure \ref{fig:max_and_not_max}. We will present a strategy to zero force any 2-connected outerplanar graph that is related to the strategy presented in \cite{hernandez2019zero}, but is more complicated to accommodate various subtleties arising from general 2-connected outerplanar graphs. Furthermore, we will demonstrate that our bounds are tight. We also provide evidence that a similar bound may hold for 3-connected planar graphs as well (Section \ref{sec:conclusion}). 


\begin{figure}[!h]
    \centering
    \begin{minipage}{.45\textwidth}
        \centering
        \begin{tikzpicture}[scale=0.4]
	\begin{pgfonlayer}{nodelayer}
		\node [style=none] (2) at (0, 2.25) {};
		\node [style=Hollow] (8) at (0, 2.25) {};
		\node [style=Hollow] (9) at (6, -1) {};
		\node [style=Hollow] (12) at (-4, -4.5) {};
		\node [style=Hollow] (17) at (-2, 5.25) {};
		\node [style=ZFS] (19) at (0, 2.25) {};
		\node [style=ZFS] (20) at (4, -4.5) {};
		\node [style=ZFS] (21) at (0, -4.5) {};
		\node [style=ZFS] (22) at (-4, 2.25) {};
		\node [style=Hollow] (23) at (0, 2.25) {};
		\node [style=ZFS] (24) at (4, 2.25) {};
		\node [style=Hollow] (28) at (2, -1) {};
		\node [style=ZFS] (30) at (2, 5.25) {};
		\node [style=ZFS] (31) at (-6, -1) {};
		\node [style=Hollow] (32) at (0, 2.25) {};
		\node [style=Hollow] (34) at (-2, -1) {};
	\end{pgfonlayer}
	\begin{pgfonlayer}{edgelayer}
		\draw (12) to (21);
		\draw (22) to (17);
		\draw (24) to (28);
		\draw (24) to (9);
		\draw (9) to (28);
		\draw (30) to (24);
		\draw (31) to (22);
		\draw (30) to (23);
		\draw (23) to (28);
		\draw (17) to (32);
		\draw (32) to (34);
		\draw (22) to (34);
		\draw (22) to (32);
		\draw (34) to (21);
		\draw (34) to (28);
		\draw (28) to (20);
		\draw (28) to (21);
		\draw (12) to (34);
		\draw (31) to (34);
		\draw (24) to (32);
		\draw (21) to (20);
	\end{pgfonlayer}
\end{tikzpicture}
        \subcaption{This is a maximal outerplanar graph with a ZFS in blue.}
    \end{minipage}%
    \begin{minipage}{0.1\textwidth}
        \hspace{0mm}
    \end{minipage}%
    \begin{minipage}{.45\textwidth}
        \centering
        \begin{tikzpicture}[scale=0.4]
	\begin{pgfonlayer}{nodelayer}
		\node [style=none] (5) at (-2, -3.25) {};
		\node [style=none] (6) at (6, 3.25) {};
		\node [style=none] (9) at (2, 3.25) {};
		\node [style=none] (13) at (-6, 3.25) {};
		\node [style=none] (16) at (-4, 0) {};
		\node [style=none] (18) at (2, -3.25) {};
		\node [style=none] (19) at (-4, -6.5) {};
		\node [style=none] (21) at (0, -6.5) {};
		\node [style=Hollow] (23) at (2, 3.25) {};
		\node [style=Hollow] (26) at (8, 0) {};
		\node [style=Hollow] (28) at (2, -3.25) {};
		\node [style=Hollow] (29) at (0, -6.5) {};
		\node [style=Hollow] (31) at (-4, -6.5) {};
		\node [style=Hollow] (32) at (-2, -3.25) {};
		\node [style=Hollow] (37) at (-6, 3.25) {};
		\node [style=Hollow] (38) at (-4, 0) {};
		\node [style=Hollow] (40) at (6, 3.25) {};
		\node [style=Hollow] (41) at (-4, 6.25) {};
		\node [style=Hollow] (46) at (4, -6.5) {};
		\node [style=ZFS] (48) at (2, 3.25) {};
		\node [style=ZFS] (49) at (4, -6.5) {};
		\node [style=ZFS] (50) at (0, -6.5) {};
		\node [style=ZFS] (52) at (-6, 3.25) {};
		\node [style=Hollow] (53) at (2, 3.25) {};
		\node [style=ZFS] (54) at (6, 3.25) {};
		\node [style=Hollow] (64) at (0, 3.25) {};
		\node [style=Hollow] (65) at (-2, 3.25) {};
		\node [style=Hollow] (67) at (-3, -1.625) {};
		\node [style=Hollow] (68) at (4, 0) {};
		\node [style=Hollow] (69) at (3, -1.6) {};
		\node [style=ZFS] (70) at (4, 6.25) {};
		\node [style=ZFS] (71) at (-8, 0) {};
	\end{pgfonlayer}
	\begin{pgfonlayer}{edgelayer}
		\draw (13.center) to (16.center);
		\draw (9.center) to (6.center);
		\draw (18.center) to (21.center);
		\draw (19.center) to (5.center);
		\draw (28) to (46);
		\draw (46) to (29);
		\draw (65) to (64);
		\draw (67) to (32);
		\draw (68) to (69);
		\draw (31) to (50);
		\draw (32) to (50);
		\draw (65) to (52);
		\draw (41) to (65);
		\draw (52) to (41);
		\draw (54) to (68);
		\draw (54) to (26);
		\draw (26) to (68);
		\draw (70) to (54);
		\draw (71) to (52);
		\draw (65) to (38);
		\draw (38) to (67);
		\draw (38) to (71);
		\draw (70) to (53);
		\draw (53) to (64);
		\draw (53) to (68);
		\draw (69) to (28);
		\draw (28) to (32);
	\end{pgfonlayer}
\end{tikzpicture}
        \subcaption{This graph is no longer maximal; one triangle has been expanded into cycle with 9 vertices. The blue vertices no longer produce a ZFS for this graph.}
    \end{minipage}
    \caption{This is a basic example showing how forcing 2-connected outerplanar graphs does not reduce to the maximal case.}
    \label{fig:max_and_not_max}
\end{figure}
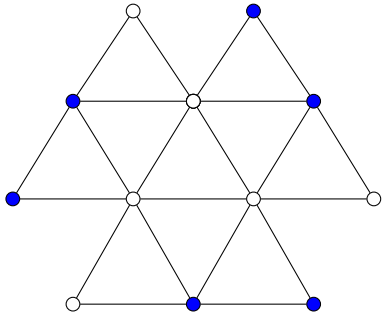
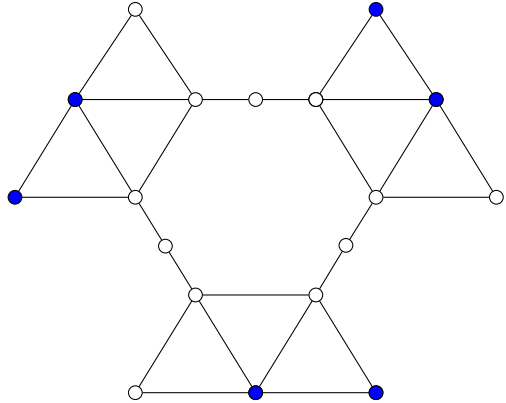

The assumptions of 2-connectedness and outerplanarity are necessary conditions. For $G$ a star (Figure \ref{fig: not_this_not_that}), which is not 2-connected, we have $Z(G) = n - 2$.

We will see that for any 2-connected outerplanar graph, the weak dual is a tree. This is an important characteristic of these types of graphs that we utilize in our zero forcing strategy. When a graph is not outerplanar, the weak dual may have cycles. This disrupts how we want to zero force and in fact may cause $Z(G)$ to be much greater than $\frac{n}{2}$. In the graph on the right in Figure \ref{fig: not_this_not_that}, we have $Z(G) = n - 2$.

\begin{figure}[H]
    \centering
    \begin{minipage}{.45\textwidth}
        \centering
        \ctikzfig{drawings/star}
        \subcaption{A graph that is outerplanar but not 2-connected. The blue vertices produce a ZFS.}
    \end{minipage}%
    \begin{minipage}{0.1\textwidth}
        \hspace{0mm}
    \end{minipage}%
    \begin{minipage}{.45\textwidth}
        \centering
        \ctikzfig{drawings/notouterplanarNEW}
        \subcaption{A graph that is 2-connected but not outerplanar. The blue vertices produce a ZFS.}
    \end{minipage}
    \caption{}
    \label{fig: not_this_not_that}
\end{figure}

\section{Preliminaries}


For a graph, the \emph{path-cover number of $G$} is the minimum size of the partition of $V(G)$ such that the induced subgraph on each part is a path (where a single vertex is considered a path).

\begin{proposition}[See \cite{work2008zero}]
For any graph $G$, $P(G) \le Z(G)$
\end{proposition}


A graph is $k$-connected if for any set of vertices $S$ with $|S| < k$, the graph induced on $V(G) \setminus S$ is connected.

A graph is \emph{outerplanar} if the graph has a planar embedding $\Gamma$ such that one face (``an outer face'') is incident to all vertices. Equivalently, a graph is outerplanar if and only if it has no $K_4$ nor $K_{2,3}$ minor \cite{Chartrand1967}.


The \emph{dual} of a planar graph $G$ with planar embedding $\Gamma$ is the graph $G^{'} = (V^{'},E^{'})$ where $V^{'}$ are the regions enclosed by $\Gamma$ and $E^{'}$ are pairs of regions which share a coincident edge. The \emph{weak dual} of an outerplanar graph $G$, denoted $G^{*}$ is the dual graph with the vertex corresponding to the outer face removed. We emphasize that  the same planar or outerplanar graph may have different nonisomorphic duals or weak duals. However, for our purposes we implicitly choose an embedding, and hence, refer to ``the'' dual or weak dual where appropriate.


\begin{proposition}[See \cite{proskurowski1981minimum}]
    For any 2-connected outerplanar graph $G$, its weak dual $G^{*}$, is a tree.
\end{proposition}


Since for 2-connected outerplanar graphs, $G^*$ is a tree, we can consider the number of leaves of $G^*$ which we will denote $n_\ell$. Note that when the outerplanar graph is maximal as in the result from (\ref{eq:maximal}), the $n_2$ in that equation corresponds to $n_\ell$.  The faces in $G$ corresponding to leaves in $G^*$ we will refer to as \emph{leaf faces} or \emph{leaf cycles}. 


A \emph{branch vertex} in $G^*$ is a degree two vertex such that it is adjacent to a leaf vertex or another branch vertex. A branch in $G^*$ is a component of the subgraph of $G^*$ induced by all branch vertices. We define a \emph{branch} of $G$ as the subgraph of $G$ given by the vertices and edges of the faces in $G$ that correspond to branch vertices in $G^*$.


A \emph{limb vertex} in $G^*$ is a vertex of degree three or more such that it is adjacent to at least one leaf or branch vertex. A \emph{limb face} or \emph{limb cycle} is a face in $G$ corresponding to a limb vertex in $G^*$. A \emph{limb} is the subgraph of $G$ corresponding to a component of the subgraph of limb vertices in $G^*$.


The \emph{foliage} of a 2-connected outerplanar graph is the subgraph of $G$ induced by the faces corresponding to all leaf, branch, and limb vertices in $G^*$.


A \emph{trunk vertex} in $G^*$ is a vertex of degree two or more that is not adjacent to a leaf nor a branch. The \emph{trunk} of $G$ is the subgraph induced by the faces in $G$ corresponding to trunk vertices in $G^*$. A \emph{trunk cycle} is the subgraph of $G$ corresponding to a component of the trunk of $G^*$.

The set of vertices in the intersection of the foliage and the trunk we will call the \emph{boundary vertices}. Every vertex in a tree is exactly one of the following:
a leaf, a branch, a limb, or a trunk vertex. Since For any 2-connected outerplanar graph $G$, its weak dual $G^{*}$, is a tree \cite{proskurowski1981minimum}, every vertex in $G$ is on either a leaf cycle, a branch cycle, a limb cycle or a trunk cycle. Hence, the union of the foliage with the trunk gives us all of $G$.

\begin{proposition}
    Let $n$ be the number of vertices of a 2-connected outerplanar graph $G$. Let $n_F$ be the number of vertices in the foliage, $n_T$ be the number of vertices in the trunk, and $n_B$ be the number of boundary vertices. Then, 
    \[ n = n_F + n_T - n_B.\]
\end{proposition} 

An \emph{outer limb} is a component of the foliage of $G$ such that it is connected to exactly one trunk cycle in $G$. Each outer limb will have exactly two boundary vertices. An \emph{inner limb} is a component of the foliage that is adjacent to at least two trunk cycles in $G$.  Each inner limb will have two boundary vertices on each trunk it is connected to. On an inner limb, the group of vertices and edges moving along the outer face of the graph from one boundary vertex to the next boundary vertex (not on the same trunk cycle) is called an \emph{inner limb segment}.

\begin{figure}
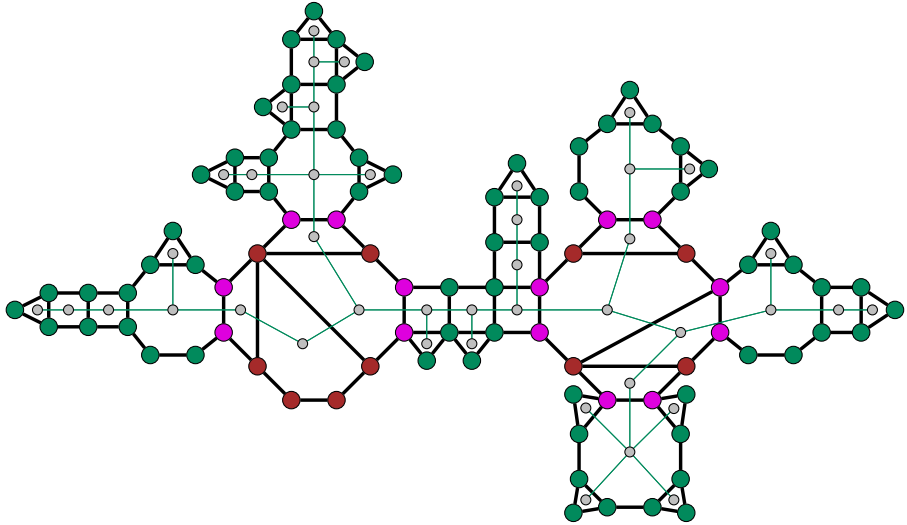

    \centering
    \ctikzfig{drawings/Foliage_Trunk_Boundary}
    \caption{A 2-connected outerplanar graph with its weak dual (in gray). The green vertices are in the foliage of $G$. The brown vertices are in the trunk of $G$. The pink vertices are the boundary vertices; they are in both the foliage and the trunk. This graph has two trunk cycles, five outer limbs, and one inner limb.}
    \label{fig:Foliage_Trunk_Boundary}
\end{figure}

\section{Bounds on $Z(G)$}

\begin{theorem}\label{thm:lower_bound}
If $G$ is a 2-connected outerplanar graph, then $Z(G)\geq n_\ell$ where $n_\ell$ denotes the number of leaves of the weak dual of $G$.
\end{theorem}
\begin{proof}
    By Proposition 1, $Z(G)\geq P(G)$.  Observe that in any path cover, to be an induced path, no path in the cover can include all the vertices of any cycle.  Thus, any leaf cycle must participate in at least 2 endpoints of paths, and each path has 2 endpoints, yielding that $P(G)\geq n_\ell$, giving us the result.
\end{proof}

The remainder of this section will be dedicated to proving an upper bound on the zero forcing number of 2-connected outerplanar graphs, which we state in the following theorem.

\begin{theorem}\label{thm:ubound}
    If $G$ is a 2-connected outerplanar graph on $n$ vertices, then \[Z(G)\leq \frac{n}{2}.\]
\end{theorem}

To prove this, we will prove the following proposition.

\begin{proposition} \label{prop:upper_bound}
 If $G$ is a 2-connected outerplanar graph, then 
 \[Z(G)\leq 2\ceil{\frac{n_\ell}{2}}+\frac{n_T-\frac{n_B}{2}}{2}\] where $n_\ell$ is the number of leaf cycles, $n_T$ is the number of vertices in the trunk, and $n_B$ is the number of boundary vertices in $G$.
\end{proposition}

Once this is proven, we will demonstrate how Theorem \ref{thm:ubound} follows from Proposition \ref{prop:upper_bound}. We prove the bound by presenting a strategy for zero forcing any 2-connected outerplanar graph.

A graph $G$ is a \textit{graph of two parallel paths} if there exist two independent induced paths of $G$ that cover all the vertices of $G$ and such that the graph can be drawn in the plane in such a way that the paths are parallel and edges (drawn as segments, not curves) between the two paths do not cross (see \cite{row2012technique}).

\begin{proposition}[See \cite{row2012technique}]
 Let $G = (V_G, E_G)$ be a graph. Then $Z(G) = 2$ if and only if $G$ is a graph of two parallel paths.
\end{proposition}


\begin{corollary} \label{Branch}
For any 2-connected outerplanar graph with $n_\ell = 2$, $Z(G) = 2$.  
\end{corollary}

\begin{proof}
    Any 2-connected outerplanar graph with $n_\ell = 2$ will be a graph of two parallel paths. Hence, $Z(G) = 2$.
\end{proof}

\begin{figure}[!h]
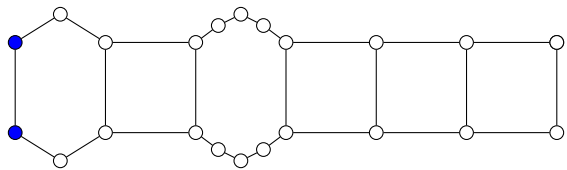

    \centering
    \ctikzfig{drawings/dualispath}
    \caption{A 2-connected outerplanar graph with $n_\ell = 2$. A ZFS is obtained by coloring two adjacent vertices on either of the two leaf cycles.}
    \label{fig: Branch}
\end{figure}

\pagebreak

\begin{claim} \label{NoTrunk}
    Every 2-connected outerplanar graph with no trunk cycles has \[ Z(G) \leq 2\left\lceil \frac{n_\ell}{2}\right\rceil.\]
\end{claim}

\begin{proof}
    We prove this claim by induction on $n_\ell$. The case when $n_\ell = 1$ is clear. The case when $n_\ell =2$ is handled by Corollary \ref{Branch}. Suppose $n_\ell$ is at least 3. 
    
    We know the dual graph is a tree. Furthermore, since there are no trunk cycles, every vertex of $G^*$ is either a leaf, a branch, or a limb vertex. The subgraph in $G^*$ formed by the limb vertices must also be a tree. If this was not the case, then there would be a trunk cycle. At least two of the limb vertices must have degree 1 in this subgraph. Since limb vertices have at least degree 3 in $G^*$, in $G^*$ these two limb vertices must be adjacent to at least two branches. 

    Pull off one of the branches in $G$ attached to one of the limb cycles described above. By ``pull off'' we mean to disconnect the branch from the remainder of the graph at the 2-separation between the branch and the limb cycle it is connected to. Include the two vertices at the 2-separation in both graphs. Notice that that the branch pulled off is now a partial 2-path and there are still no trunk vertices in the remainder of the graph. Let $G'$ be the graph with the branch removed and $n_\ell'$ be the number of leaf cycles of $G'$. Then $n_\ell' = n_\ell - 1$.

    Our forcing strategy is to color two adjacent vertices on every other leaf cycle. By every other, we mean $\left\lceil \frac{n_\ell}{2}\right\rceil$ leaf cycles.
    
    Case 1: $n_\ell$ is even. Thus $G'$ has an odd number of leaves. By the induction hypothesis, a zero forcing set for $G'$ is accomplished by coloring 2 vertices of the leaves of every other branch, including two branches that are right next to each other. Without loss of generality, suppose that the removed branch is between these two branches. We claim that the zero forcing set for $G'$ also is a zero forcing set for $G$. This set can force down the branches that were colored to the limb cycles. Then, since both branches that were closest to the removed branch were colored, this can force along the limb cycles up to the 2 vertices where the branch was pulled off. Then the remainder of the forces in $G'$ can proceed unobstructed. What remains in $G$ is a single branch with two of its adjacent endpoints colored. These can then force along the branch since the branch is a partial 2-path.
    
    Case 2: $n_\ell$ is odd. Thus $G'$ has an even number of leaves. By the induction hypothesis, $G'$ can be forced by coloring every other branch. Now in $G$, color these vertices of $G'$ as well as two adjacent nodes of the leaf of the removed branch. These can force down its branch to the 2 vertices that it was pulled off of. Then the remainder of the colored nodes can force the rest of $G'$ without obstruction.
\end{proof}

\begin{figure}[!h]
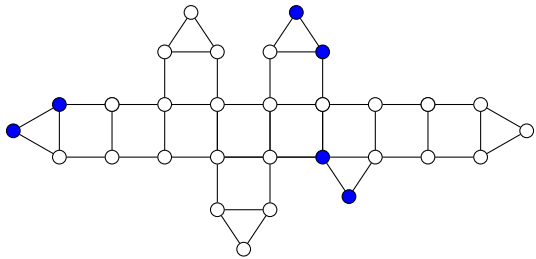

    \centering
    \ctikzfig{drawings/NoTrunk_ZFS}
    \caption{A 2-connected outerplanar graph with no trunk. Coloring two adjacent vertices on every other leaf produces a Zero Forcing Set.}
    \label{fig: No_Trunk}
\end{figure}

\begin{claim} \label{BaseCase1}
    For a 2-connected outerplanar graph with one trunk cycle a ZFS is produced by coloring two adjacent vertices on every other leaf cycle and coloring less than or equal to half the number of not-initially-colored vertices of the trunk. 
\end{claim}

In order to prove this fact, we will need to understand how forcing works both in the foliage and in the trunk. Lemma \ref{lem: OuterLimbs} shows how forcing operates in outer limbs.

Outer limbs are classified in three ways. They are classified based whether there are an even or odd number of leaf cycles on the outer limb and how many of these limbs have vertices in our proposed ZFS. Outer limbs that have an even number of leaf cycles with exactly half of those leaves having vertices in our proposed ZFS we call \emph{$E$ limbs}. Outer limbs with an odd number of leaves with one more than half of the leaves having vertices in our proposed ZFS we call \emph{$O^{+}$ limbs}. Finally, Outer limbs with an odd number of leaves with one fewer than half of the leaves having vertices in our proposed ZFS we call \emph{$O^{-}$ limbs}.

\begin{lemma} \label{lem: OuterLimbs}
    Color two adjacent vertices on every other leaf cycle in $G$. The following will occur:
    \begin{enumerate}
        \item Outer $E$ limbs will force one boundary vertex. If the other boundary vertex gets forced from the trunk cycle, then the entire limb will get forced.
        \item Outer $O^{+}$ limbs will force the entire limb, through the boundary vertices, and onto the trunk.
        \item Outer $O^{-}$ limbs will force neither boundary vertex. If both boundary vertices get forced from the trunk cycle, then the entire inner limb will get forced.
    \end{enumerate}
\end{lemma}

\begin{proof}
    Begin by coloring two adjacent vertices on every other leaf cycle. Each outer limb has exactly two boundary vertices. Remove the outer limb from the trunk at and including the boundary vertices. Then, attach a cycle with no chords (a leaf cycle) to the boundary vertices. The resulting graph, call it $G'$, will have no trunk. Let $n_\ell'$ be the number of leaf cycles in $G'$. Notice that in $G$ there were $n_\ell' - 1$ leaf cycles on this outer limb. There are three possibilities of what may happen:

    1. Outer $E$ limb: $n_\ell'$ is odd and $\frac{n_\ell' - 1}{2}$ of the leaves are colored. We know that since we colored the leaves before adding the leaf cycle on the boundary vertices that the added leaf will not be colored. Since we colored every other leaf, we know that exactly one of the two leaves closest to the added leaf will be colored. By the same strategy presented in Claim \ref{NoTrunk}, the colored leaves will force down their branches, along the limb cycles, and to the bases of the closest branches. Exactly one of the boundary vertices will get forced. By Claim \ref{NoTrunk}, if the added leaf were to be colored, then all of $G'$ would be forced. Consider removing the added leaf cycle and attaching the limb back on the trunk in $G$. One of the boundary vertices will be forced and if the other boundary vertex gets forced from the trunk, then it will be as if the added leaf was colored in $G'$. The entire limb would get forced.

    2. Outer $O^{+}$ limb: $n_\ell'$ is even and $\frac{n_\ell'}{2}$ of the leaves are colored. By Claim \ref{NoTrunk}, all of $G'$ will be forced. We know the added leaf cycle is not originally colored. Thus, it gets forced from the ZFS. Therefore, in $G$, the limb will be entirely forced and will force through both boundary vertices onto the trunk. 

    3. Outer $O^{-}$ limb: $n_\ell'$ is even and $\frac{n_\ell'}{2} - 1$ of the leaves are colored. Color two adjacent vertices on the added leaf cycle. By Claim \ref{NoTrunk}, all of $G'$ will get forced. Thus, when both boundary vertices of an outer $O^{-}$ limb are forced from along the trunk cycle, the entire limb will get forced.
\end{proof}

Now, we will use this lemma to help us prove how to force a 2-connected outerplanar graph with one trunk cycle.

\noindent \textit{Proof of Claim \ref{BaseCase1}.}
    There are two cases that need to be considered.

    Case 1: There is at least one odd-leafed limb. Begin by coloring two adjacent vertices on every other leaf cycle, starting with an odd-leafed outer limb and coloring it so it is an $O^{+}$ limb. If $n_\ell$ is odd, then there will be one more $O^{+}$ limb than $O^{-}$ limb. Otherwise, there will be an equal number of $O^{+}$ and $O^{-}$ outer limbs. Since $O^{+}$ limbs force completely and begin forcing along the trunk cycle, it is easier to force a graph when there are more $O^{+}$ outer limbs. We will achieve our desired result by showing our strategy must hold for graphs with an equal number of $O^{+}$ and $O^{-}$ outer limbs. In other words, when $n_\ell$ is even. 
    
    Suppose $n_\ell$ is even. From lemma \ref{lem: OuterLimbs} we understand that $O^{+}$ force completely and begin forcing along the trunk while $O^{-}$ need to be forced from the trunk (see Figure \ref{fig:one_trunk}). By virtue of every other leaf being colored, $O^{+}$ and $O^{-}$ limbs will oscillate around the outerface of the graph. The goal is to ensure that each $O^{+}$ limb can force along both directions around the outerface of the trunk cycle until it reaches a boundary vertex of its neighboring $O^{-}$ limbs. Since forcing passes through $E$ limbs, the only potential hindrance to forcing around the trunk cycle are the trunk cycle's chords.
    
    Choose any $O^{+}$ limb. Then, select a direction---clockwise or counter clockwise. If there are no chord ends between the $O^{+}$ limb and the next $O^{-}$ limb then all the vertices of the trunk between these two limbs will get forced. If there is a chord end between these two limbs, then forcing will stop at this chord end. This vertex may be the chord end to multiple chords. Each of the other chord ends must be colored in order for forcing to continue along the trunk cycle. There are three types of chord configurations that must be addressed (see Figure \ref{fig:chord_cases}).

    The first case occurs when both chords ends are only the ends to a single chord. In this case, color the other chord end and forcing will continue. Since we only have to color one of the two chord ends, these types of chords will never make us color more than half of the not-initially-colored vertices of the trunk cycle.

    The second case occurs when both chord ends are the ends to multiple chords. The existence of chords like this requires there to be some outer limb in between these chord ends. If this were not the case, then some face of the trunk would have degree one which by definition cannot occur. We do not need to color the not-initially-colored boundary vertices (unless perhaps they are also chord ends). For each chord end we need to color we can pair it with one of these not-initially-colored boundary vertices. This ensures that we color no more than half of the not-initially-colored vertices of the trunk cycle.

    The third case occurs when the chords are arranged in a fan-like structure (see Figures \ref{fig:chord_cases}(c), \ref{fig:chord_cases}(f), and \ref{fig:reverse_chords}). Each of the tips of the fan (the other chord ends) must be colored in order for forcing to continue. It is possible that forcing through a fan requires that we color more than half of the not-initially-colored vertices of the trunk. If this is the case, we must reverse the direction forcing occurs around the trunk cycle by altering the leaf cycles that we colored. Switching the leaf cycles that are colored changes all $O^{+}$ limbs to $O^{-}$ limbs and vice versa. This reverses the direction of forcing along the trunk cycle.. 

    When the direction of forcing around the trunk cycle is reversed, the chord ends of a fan that need to be colored are exactly reversed as well. This ensures that in one of the two directions, we are guaranteed to color no more than half of the not-initially-colored vertices of the trunk cycle.

    Now that we are able to force through all the chord ends, each $O^{+}$ limb will be able to force around the trunk cycle to its neighboring $O^{-}$ limbs. Each $O^{-}$ limb will have both of its boundary vertices forced. Thus, by Lemma \ref{lem: OuterLimbs}, each $O^{-}$ limb will get forced completely and the entire graph will be forced. 

    Case 2: There are no odd-leafed limbs. Color two adjacent vertices on every other leaf. There are only $E$ limbs and $n_\ell$ must be even. Since $E$ limbs do not force along the trunk cycle unless one boundary vertex is forced from the trunk cycle, we must color one boundary vertex of an $E$ limb to begin forcing along the trunk. Choose any limb and add the boundary vertex not forced from the leaves to the ZFS. Then forcing will begin around one direction of the trunk. Plug each chord as needed (in the same manner as in Case 1) such that we can force to the next $E$ limb. One of the next $E$ limb's boundary vertices is forced by its leaves. The other will be forced from the trunk cycle. That entire limb will get forced and then we continue to plug each chord as needed until we reach the next $E$ limb. Proceed by plugging each chord end until we get back to the first $E$ limb. Finally, the boundary vertex we added to the ZFS will be able to force up into the outer $E$ limb and the entire graph will be forced. \qed

\begin{figure}
    \centering
    \begin{minipage}{.45\textwidth}
       \centering
       \ctikzfig{drawings/OneTrunk_Leaves}
       \subcaption{A 2-connected outerplanar graph with one trunk cycle and three outer limbs. Two adjacent vertices are colored on every other leaf cycle.}
    \end{minipage}%
    \begin{minipage}{.1\textwidth}
        \hspace{0mm}
    \end{minipage}%
    \begin{minipage}{.45\textwidth}
       \centering
       \ctikzfig{drawings/OneTrunk_Stuck}
       \subcaption{The blue vertices are those that were originally colored and the red are the vertices that were forced by the blue vertices. The $O^{+}$ limb forced completely, the $E$ limb forced one of two boundary vertices, and the $O^{-}$ limb did not force either boundary vertex.}
    \end{minipage}%
    \vspace{2mm}
    \begin{minipage}{0.45\textwidth}
       \centering
       \ctikzfig{drawings/OneTrunk_Unstuck}
       \subcaption{Two chord ends of the trunk cycle were colored blue. Now, forcing can continue counter-clockwise around the trunk cycle, passing through the $E$ limb, and continuing to the $O^{-}$ limb. Both boundary vertices of the $O^{-}$ limb will be forced, hence it will force completely. The blue vertices form a ZFS.}
    \end{minipage}
    \caption{}
    \label{fig:one_trunk}
\end{figure}

\begin{figure}
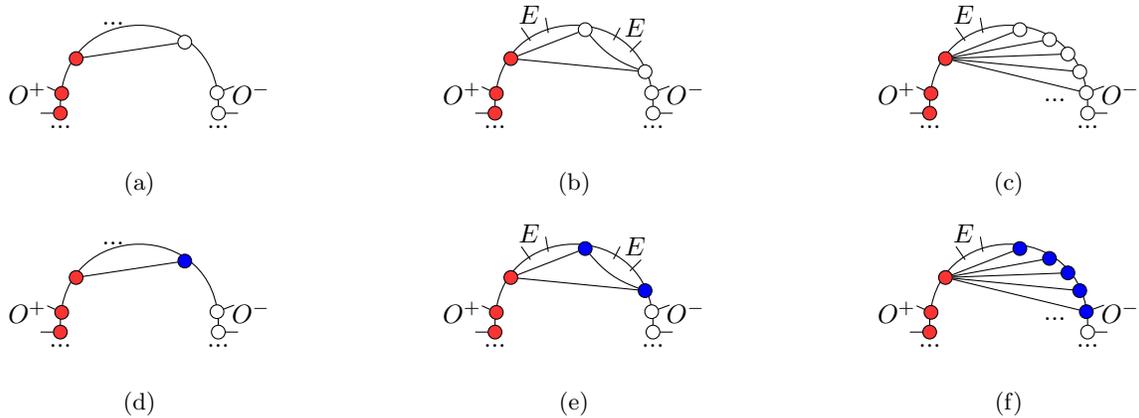

    \centering
    \begin{minipage}{0.30\textwidth}
        \centering
        \ctikzfig{drawings/ex1_reversal_chords}
        \subcaption{}
    \end{minipage}%
    \begin{minipage}{0.05\textwidth}
        \hspace{0mm}
    \end{minipage}%
    \begin{minipage}{0.30\textwidth}
        \centering
        \ctikzfig{drawings/ex2_reversal_chords}
        \subcaption{}
    \end{minipage}%
    \begin{minipage}{0.05\textwidth}
        \hspace{0mm}
    \end{minipage}%
    \begin{minipage}{0.30\textwidth}
        \centering
        \ctikzfig{drawings/ex3_reversal_chords}
        \subcaption{}
    \end{minipage}%
    
    \vspace{2mm}

    \begin{minipage}{0.30\textwidth}
        \centering
        \ctikzfig{drawings/ex1_reversal_chords_ZFS}
        \subcaption{}
    \end{minipage}%
    \begin{minipage}{0.05\textwidth}
        \hspace{0mm}
    \end{minipage}%
    \begin{minipage}{0.30\textwidth}
        \centering
        \ctikzfig{drawings/ex2_reversal_chords_ZFS}
        \subcaption{}
    \end{minipage}%
    \begin{minipage}{0.05\textwidth}
        \hspace{0mm}
    \end{minipage}%
    \begin{minipage}{0.30\textwidth}
        \centering
        \ctikzfig{drawings/ex3_reversal_chords_ZFS}
        \subcaption{}
    \end{minipage}
    
    \caption{Each figure shows a section of a trunk cycle between an $O^{+}$ limb and an $O^{-}$ limb. Forcing begins at the $O^{+}$ limb and works clockwise around the trunk cycle to the $O^{-}$ limb. There may be any number of degree two vertices on the trunk cycle. Only the boundary vertices and chord ends are shown. The specific structure of the outer limbs is omitted.}
    \label{fig:chord_cases}
\end{figure}

\begin{figure}
    \centering
    \begin{minipage}{0.30\textwidth}
       \centering
        \ctikzfig{drawings/Reverse_Chords}
        \subcaption{A 2-connected outerplanar graph with one trunk cycle. We have colored every other leaf.} 
    \end{minipage}%
    \begin{minipage}{.05\textwidth}
        \hspace{0mm}
    \end{minipage}%
    \begin{minipage}{0.30\textwidth}
        \centering
        \ctikzfig{drawings/Reverse_Chords_Stuck}
        \subcaption{The leaves force as much as they can.}
    \end{minipage}%
    \begin{minipage}{.05\textwidth}
        \hspace{0mm}
    \end{minipage}%
    \begin{minipage}{0.30\textwidth}
        \centering
        \ctikzfig{drawings/Reverse_Chords_Unstuck}
        \subcaption{We color the chord ends necessary in order to continue to force counter-clockwise from the $O^{+}$ limb to the $O^{-}$ limb.}
    \end{minipage}%

    \vspace{2mm}
    \begin{minipage}{0.30\textwidth}
        \centering
        \ctikzfig{drawings/Reverse_Chords_Fixed}
        \subcaption{By switching the leaves colored, we switch the direction forcing occurs around the trunk.}
    \end{minipage}%
    \begin{minipage}{.05\textwidth}
        \hspace{0mm}
    \end{minipage}%
    \begin{minipage}{0.30\textwidth}
        \centering
        \ctikzfig{drawings/Reverse_Chords_Fixed_Stuck}
        \subcaption{We force from the leaf cycles as much as we can.}
    \end{minipage}%
    \begin{minipage}{.05\textwidth}
        \hspace{0mm}
    \end{minipage}%
    \begin{minipage}{0.30\textwidth}
        \centering
        \ctikzfig{drawings/Reverse_Chords_Fixed_ZFS}
        \subcaption{Now, we only need to color one chord end to continuing forcing along the trunk clockwise from the $O^{+}$ limb to the $O^{-}$ limb.}
    \end{minipage}
    \caption{}
    \label{fig:reverse_chords}
\end{figure}

Just like outer limbs, inner limb segments are characterized by the number of leaves and the amount of leaves that have vertices in the proposed ZFS; we can have $E$, $O^{+}$, and $O^{-}$ inner limb segments. The next three lemmas will be useful for considering how forcing occurs in inner limbs. This will be crucial to know in order to prove our result for graphs with multiple trunk cycles.

\begin{lemma} \label{lem: Odd inner}
    Every $O^{+}$ inner limb segment will force both of its boundary vertices.
\end{lemma}

\begin{proof}
Consider an inner limb that has an $O^{+}$ segment. Remove the trunks off of the inner limb (keeping the boundary vertices) and add leaf cycles where the trunks were removed. Then, we have a 2-connected outerplanar graph with no trunk. We know that colored leaves will force down their branches and then force along the outside of the graph until the base of the next branch (or leaf). In an $O^{+}$ segment, the two leaves closest to the boundary vertices will be colored. They will force down their respective branches and then to the base of the leaf cycles closest to them. These will be the leaf cycles we added on to the boundary vertices. Thus, the leaf cycles in the $O^{+}$ segment will force to the boundary vertices. 
\end{proof}

\begin{lemma} \label{lem: diagonal E bridge}
    Suppose there is an inner limb with exactly two $E$ segments. If one boundary vertex is forced on both trunk cycles such that they are diagonal to each other, then the entire inner limb will get forced.
\end{lemma}

\begin{proof}
Remove both trunk cycles that are connected to the inner limb with two even segments. Then, we are left with just the inner limb, which is a 2-connected outerplanar graph with no trunk. Exactly half of its leaves are colored, so we know it can be forced completely. The forcing works precisely the same when the trunk cycles are put back on the inner limb, assuming that two diagonal boundary vertices are forced from the trunk cycles.
\end{proof}

\begin{figure}[H]
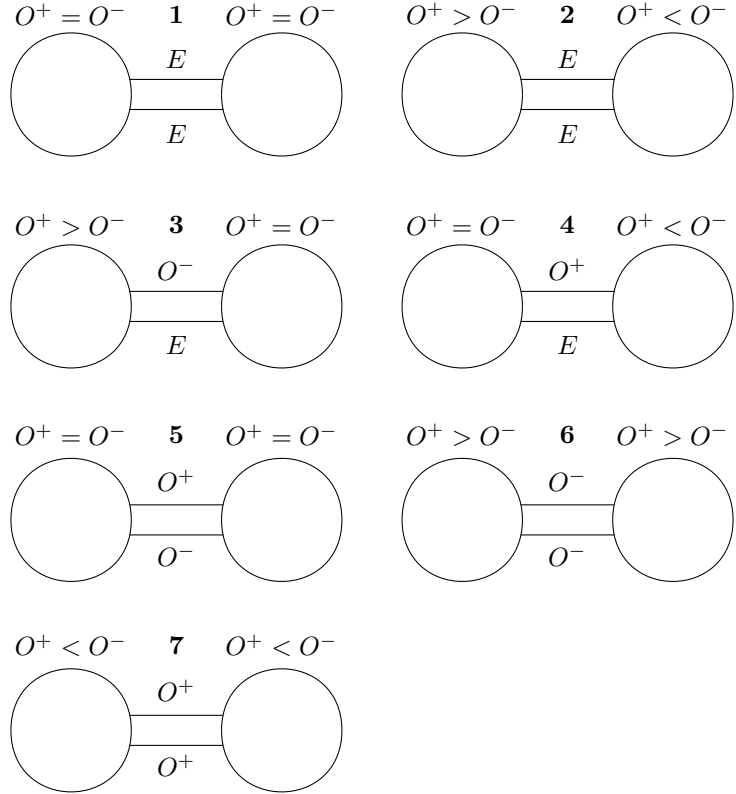

    \ctikzfig{drawings/TwoTrunks_Template}
    \caption{Each graph is a 2-connected outerplanar graph with two trunks cycles. The chords within the trunks, outer limbs, and leaves on the inner limb are omitted.}
    \label{fig: Seven Cases}
\end{figure}

\begin{claim} \label{TwoTrunks}
    Let $G$ be a 2-connected outerplanar graph with exactly two trunk cycles. Then a ZFS is produced by coloring two adjacent vertices on every other leaf cycle and coloring less than or equal to half the number of not-initially-colored vertices of the trunk.
\end{claim}

\begin{proof}
    We will consider seven cases (see Figure \ref{fig: Seven Cases}).

    Case 1: Forcing will begin at the outer $O^{+}$ limb cycles. They will force around the trunk cycle until they reach the neighboring outer $O^{-}$ limbs (plugging the chords as necessary). Since both trunk cycles have an equal number of $O^{+}$ and outer $O^{-}$  limbs, we know that both of them would be able to completely force  if the inner limb did not connect them. Hence, both trunk cycles will force around the trunk until a boundary vertex of the inner limb. Coloring every other leaf in the graph ensures that the $O^{+}$ and outer $O^{-}$  limbs will alternate around the outer face of the graph. Thus, the boundary vertices that are forced from the trunk cycles will be diagonal to each other. By Lemma \ref{lem: diagonal E bridge}, the inner limb will be forced completely and forcing can continue unobstructed around the remainder of the two trunk cycles.

    Case 2: The trunk cycle with one more $O^{+}$ than $O^{-}$ will force completely. We can reduce this cycle to a leaf on the inner limb. The inner limb will then act as an outer $O^{+}$  limb on the remaining trunk cycle. This trunk cycle will now have an equal number of $O^{+}$ and outer $O^{-}$  limbs and thus will be able to force completely as described in Claim \ref{BaseCase1}.

    Case 3: The trunk cycle with one more $O^{+}$ than $O^{-}$ will force entirely. We can reduce this cycle to a leaf cycle on the inner limb. The inner limb thus becomes an $E$ limb with half of its leaves colored. So by Claim \ref{BaseCase1}, the remainder of the graph will be forced (see Figure \ref{fig:twotrunks}).
    
    Case 4: The $O^{+}$ on the inner limb will force both of its boundary vertices by Lemma \ref{lem: Odd inner}. The trunk cycle with an equal number of $O^{+}$ and $O^{-}$ will force the other boundary vertex of itself with the inner limb. The entire trunk cycle will then be forced. Forcing will continue through the $E$ segment on the inner limb. Finally, we can reduce this trunk cycle and the inner limb to an $O^{+}$ on the other trunk cycle. Thus, this trunk cycle now has an equal number of $O^{+}$ and $O^{-}$ and will force completely by Claim \ref{BaseCase1}.

    Case 5: Both trunk cycles will force until the its respective boundary vertex with the $O^{-}$ segment on the inner limb. The $O^{+}$ on the inner limb will force both of its boundary vertices. The entire inner limb will get forced and then forcing can continue around the trunk cycles as described by Claim \ref{BaseCase1}.
    
    Case 6: Both trunk cycles will force completely, hence, both boundary vertices on both $O^{-}$ segments of the inner limb will get forced. We can reduce both trunk cycles to colored leaf cycles on the inner limb. The result is a 2-connected outerplanar graph with no trunk with an even number of leaves. Notice that exactly half of the leaves are colored. The remainder of the graph will force by Claim \ref{NoTrunk}.   

    Case 7: Both $O^{+}$ segments on the inner limb will force completely. Both boundary vertices on both trunk cycles will be forced. The inner limb now acts as an outer $O^{+}$  limb on each trunk cycle. Thus, it is as if both trunk cycles now have an equal number of $O^{+}$ and outer $O^{-}$  limbs. The remainder of the graph will be forced completely by Claim \ref{BaseCase1}.
\end{proof}

\begin{figure}
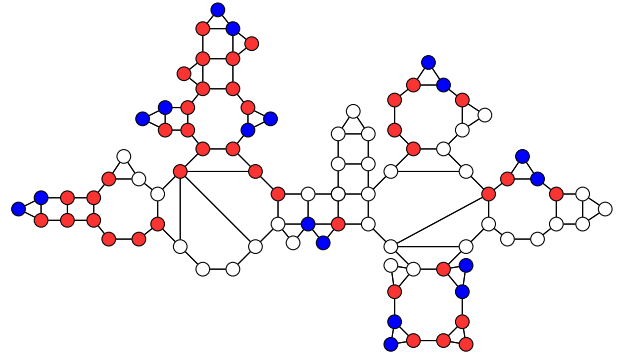
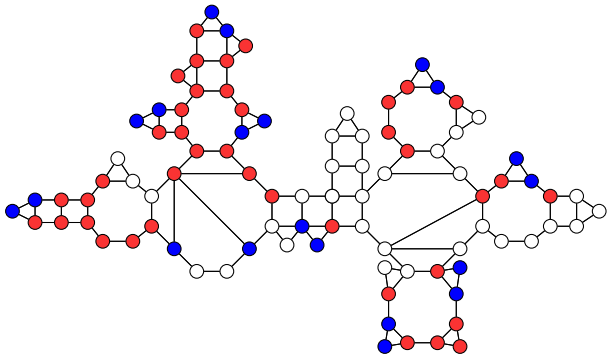
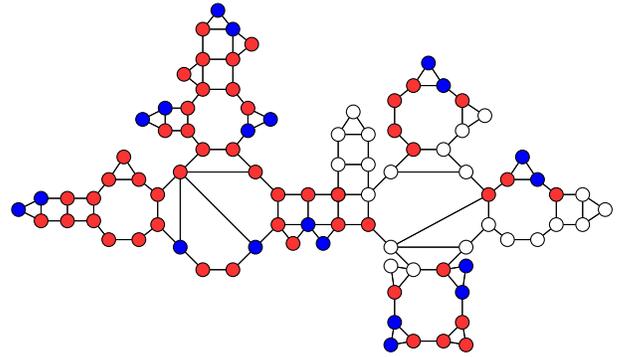
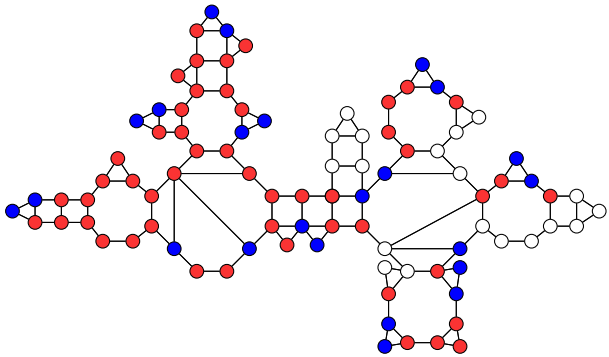
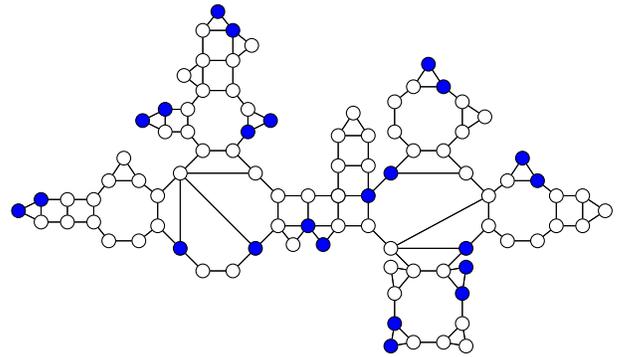

    \centering
    \begin{minipage}{.45\textwidth}
        \centering
        \ctikzfig{drawings/fg}
        \subcaption{Color every other leaf with two adjacent vertices. The trunk cycle on the left has an outer $O^{+}$ limb and an outer $E$ limb. The trunk cycle on the right has three outer $E$ limbs. The top inner limb segment is an $O^{-}$ segment and the bottom inner limb segment is an $E$ segment.}
    \end{minipage}%
    \begin{minipage}{.1\textwidth}
        \hspace{0mm}
    \end{minipage}%
    \begin{minipage}{.45\textwidth}
        \centering
        \ctikzfig{drawings/fg7stuck}
        \subcaption{The vertices on the leaves force as much as they can and are now stuck.}
    \end{minipage}%
    \vspace{1mm}
    \begin{minipage}{.45\textwidth}
        \centering
        \ctikzfig{drawings/fg8}
        \subcaption{Two vertices must be colored on the trunk on the left in order for forcing to continue along that trunk cycle.}
    \end{minipage}%
    \begin{minipage}{.1\textwidth}
        \hspace{0mm}
    \end{minipage}%
    \begin{minipage}{.45\textwidth}
        \centering
        \ctikzfig{drawings/fg14}
        \subcaption{Forcing again is halted. With the left trunk cycle and its outer limbs completely forced, this graph is now identical to forcing a graph with one trunk cycle and four $E$ outer limbs.}
    \end{minipage}%
    \vspace{1mm}
    \begin{minipage}{.45\textwidth}
        \centering
        \ctikzfig{drawings/fg14ZFSright}
        \subcaption{Three vertices must be colored on the trunk in order for forcing to go completely around the trunk cycle. The remainder of the graph will get forced.}
    \end{minipage}%
    \begin{minipage}{.1\textwidth}
        \hspace{0mm}
    \end{minipage}%
    \begin{minipage}{.45\textwidth}
        \centering
        \ctikzfig{drawings/fgZFS}
        \subcaption{This is a ZFS for our graph. This graph falls under Case 3 in the proof of Claim \ref{Inductive step}.}
    \end{minipage}
    \caption{Zero forcing a 2-connected outerplanar graph with two trunk cycles. This is the same graph that appears in Figure \ref{fig:Foliage_Trunk_Boundary}.}
    \label{fig:twotrunks}
\end{figure}

\begin{claim} \label{Inductive step}
    Let $G$ be a 2-connected outerplanar graph. Then a ZFS is produced by coloring two adjacent vertices on every other leaf cycle and coloring less than or equal to half the number of not-initially-colored vertices of the trunk.
\end{claim}

\begin{proof}
    Suppose $G$ more than one trunk cycle. Since the weak dual graph of $G$ is a tree, there are at least two trunk cycles that are connected (via an inner limb) to exactly one other trunk cycle. Select one of these trunk cycles. We are going to split $G$ into two graphs by removing the selected trunk cycle including the inner limb it is attached to. Call this graph $G''$. Now, duplicate the inner limb and attach it to the remainder of the graph and call this graph $G'$ (see Figure \ref{fig: induct_step_trunks}). We construct the graphs in this manner to ensure that the number of trunk cycles in $G'$ is exactly one less than the number of trunk cycles of $G$. The seven cases below mimic the seven cases of Claim \ref{TwoTrunks} except now the two trunk cycles are replaced by $G'$ and $G''$. 

    Case 1: The inner limb will act as an outer $E$ limb in both $G'$ and $G''$. $G'$ and $G''$ will each have an equal number of $O^{+}$ and outer $O^{-}$  limbs. By the induction hypothesis, both will force completely. Since the inner limb acts as an outer $E$ limb, we know one of its boundary vertices will be forced by the connected trunk cycle. Since the Odd limbs oscillate around $G$, the two forced boundary vertices of the inner limb in $G$ will be diagonal to each other. By Lemma \ref{lem: diagonal E bridge}, The inner limb will force completely. Then, forcing will continue in $G$ exactly as it did in $G'$ and $G''$. 
    
    Case 2: The inner limb will act as an outer $E$ limb in both $G'$ and $G''$. $G'$ will have one more $O^{+}$ than outer $O^{-}$  limb, so it will force completely. Since we have more $O^{+}$ than $O^{-}$, both boundary vertices of the inner limb can be forced from the trunk cycle. Therefore, we can think of the inner limb as having an extra leaf (already colored) in $G''$. This will make the inner limb act as an outer $O^{+}$  limb in $G''$. Then, $G''$ will have and equal amount of $O^{+}$ and outer $O^{-}$  limbs, so it will force completely.

    Case 3: $G'$ will have an equal number of $O^{+}$ and outer $O^{-}$  limbs and will thus force completely. In $G'$ the inner limb acts as an $O^{-}$. Therefore, both boundary vertices will be forced from the trunk that the inner limb is attached to. Therefore, we can think of the inner limb as having an extra leaf (already colored) in $G''$. This will make the inner limb act as an outer $E$ limb in $G''$. $G''$ then has an equal amount of $O^{+}$ and outer $O^{-}$ limbs, so it will force completely. 

    Case 4: $G'$ will have one more $O^{+}$ than $O^{-}$ and $G''$ will have an equal amount of $O^{+}$ and $O^{-}$. By the induction hypothesis, both graphs will force completely. Since the inner limb acts as an outer $O^{+}$ limb in $G'$, we know that $G'$ could force entirely without the inner limb. Hence, we know one of the boundary vertices of the inner limb can be forced by the remainder of the graph. In $G$ we know that the $O^{+}$ segment on the inner limb will force both of its boundary vertices. With a third boundary vertex forced by the trunk cycle (that is part of $G'$), the entire inner limb will force and it will then act as an $O^{+}$ for the remainder of the graph. But this is equivalent to forcing $G''$ which we know will happen. Thus, $G$ will force entirely.  

    Case 5: $G'$ and $G''$ have an equal number of $O^{+}$ and outer $O^{-}$ limbs. By the induction hypothesis, both $G'$ and $G''$ will force completely. The inner limb acts as an outer $E$ limb on both $G'$ and $G''$. Hence, one boundary vertex will be forced by the trunk cycle and the other boundary vertex will be forced by the inner limb. In $G$ the inner limb will force by Lemma \ref{lem: diagonal E bridge}. Forcing will propagate as desired through the remainder of the graph. 

    Case 6: The inner limb acts as an $O^{-}$ outer limb in both $G'$ and $G''$. In $G$ the inner limb needs all four of its boundary vertices forced from the trunk cycles it is connected to. By the induction hypothesis, both $G'$ and $G''$ have an equal number of $O^{+}$ and outer $O^{-}$ limbs. $G'$ and $G''$ will both force entirely up to the inner limb. Hence, in $G$ all four boundary vertices will be forced. Finally, the inner limb will be forced. This completes the forcing for all of $G$.

    Case 7: The inner limb acts as an outer $O^{+}$ limb cluster on both graphs. In $G$ the inner limb will force all four of its boundary vertices, thus acting as an outer $O^{+}$ limb on both trunk cycles it is attached to. By the induction hypothesis, both graphs have an even number of $O^{+}$ and outer $O^{-}$ limbs so they will force entirely. 
\end{proof}

\begin{figure}[H]
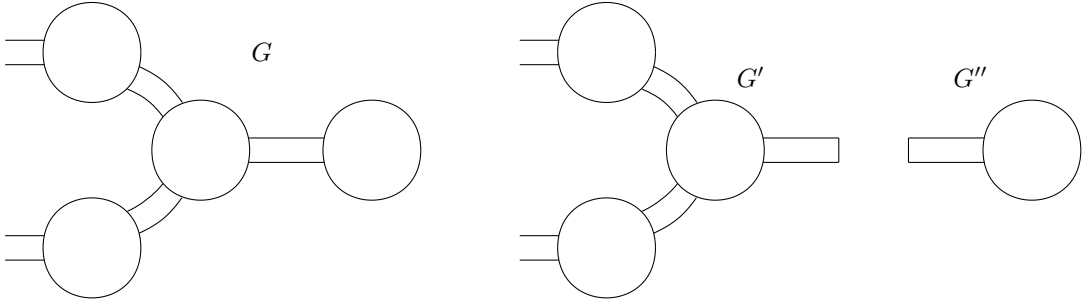

    \ctikzfig{drawings/InductiveStep_Trunks}
    \caption{$G$ is an arbitrary 2-connected outerplanar graph. The circles represent trunk cycles. The connecting pieces represent inner limbs. The outer limbs are omitted for simplicity. This shows how we split $G$ into two separate graphs $G'$ and $G''$ where both graphs include the shared inner limb which becomes an outer limb as the graphs are pulled apart.}
    \label{fig: induct_step_trunks}
\end{figure}

Proposition \ref{prop:upper_bound} follows clearly from Claim \ref{Inductive step}. We now show that Proposition \ref{prop:upper_bound} implies Theorem 2.

\begin{claim} \label{InequalityForFoliage}
    The number of leaves of the 2-connected outerplanar graph is at most $\frac{n_F - \frac{n_B}{2}}{2}$.
\end{claim}
\begin{proof}
     We create an injective function between the leaves and half the vertices of the graph by mapping a leaf to one of the degree two vertices of the cycle for that leaf.  Each cycle has at least three vertices, and two leaf cycles can overlap in at most one vertex by outerplanarity if it has more than 4 vertices. Furthermore, all leaves are part of a limb (by definition of a limb) and each boundary edge will not have a leaf, so one of its vertices is unnecessary in the injection created.
\end{proof}

\begin{claim} \label{TightBoundsToHalf}
If $G$ is a 2-connected outerplanar graph, then 
\[2\ceil{\frac{n_\ell}{2}}+\frac{n_T-\frac{n_B}{2}}{2} \leq \frac{n}{2}+1.\]
\end{claim}
\begin{proof}
Suppose $n_\ell$ is even. Then $2\ceil{\frac{n_\ell}{2}} = n_\ell$ and so by Claim \ref{InequalityForFoliage},
\[2\ceil{\frac{n_\ell}{2}}+\frac{n_T-\frac{n_B}{2}}{2} 
\leq n_\ell +\frac{n_T-\frac{n_B}{2}}{2}
\leq \frac{n_F - \frac{n_B}{2}}{2} +\frac{n_T-\frac{n_B}{2}}{2}
= \frac{n_F + n_T - n_B}{2}
= \frac{n}{2}.\]

Now, suppose $n_\ell$ is odd. Then $2\ceil{\frac{n_\ell}{2}} = n_\ell + 1$ so 
\[2\ceil{\frac{n_\ell}{2}}+\frac{n_T-\frac{n_B}{2}}{2} 
\leq n_\ell + 1 +\frac{n_T-\frac{n_B}{2}}{2}
\leq \frac{n_F - \frac{n_B}{2}}{2} +\frac{n_T-\frac{n_B}{2}}{2} + 1
= \frac{n_F + n_T - n_B}{2} + 1
= \frac{n}{2} + 1.\]

\end{proof} 

\noindent \textbf{Theorem 2}. \textit{If $G$ is a 2-connected outerplanar graph on $n$ vertices, then \[Z(G)\leq \frac{n}{2}.\]}

\noindent\textit{Proof}. From Proposition \ref{prop:upper_bound} we know that 
\begin{equation}
    Z(G) \leq 2\ceil{\frac{n_\ell}{2}}+\frac{n_T-\frac{n_B}{2}}{2}.
\end{equation}
By Claim \ref{TightBoundsToHalf}, this is bounded above by $\frac{n}{2}+1$.  Examining the proof of Claim \ref{TightBoundsToHalf}, for graphs with $n_\ell$ even we actually have \begin{equation}
    Z(G) \leq \frac{n}{2}
\end{equation} 
and the plus 1 is only necessary for graphs with $n_\ell$ odd. We will examine this case more closely.
For $n_\ell = \frac{n_F - \frac{n_B}{2}}{2}$ to be true, all leaf cycle must be triangles and there cannot be any branch cycles. Furthermore, each leaf cycles must share two of its vertices with other leaf cycles. This is what we call a sun graph (see Figure \ref{fig:sun}). When a sun graph has $n_\ell$ odd, a ZFS can be achieved by coloring $2\ceil{\frac{n_\ell}{2}} - 1$ of the vertices. Therefore, for a sun graph $G$ with $n_\ell$ odd we have 
\begin{equation}
    Z(G) = 2\ceil{\frac{n_\ell}{2}} - 1 
    = n_\ell 
    = \frac{n}{2}.
\end{equation}
Suppose we have a 2-connected outerplanar graph $G$ with no trunk cycle and $n_\ell$ odd. By Claims \ref{NoTrunk} and \ref{TightBoundsToHalf} we know that $Z(G) \leq 2\ceil{\frac{n_\ell}{2}} \leq \frac{n}{2} + 1$. If our graph $G$ is not a sun graph, then there must be either some leaf cycle that is not a triangle, a branch cycle, or an added vertex on a limb cycle. In any of these cases, this extra vertex does not need to be added to our ZFS. Therefore, we achieve $Z(G) \leq \frac{n}{2}$.

For any 2-connected outerplanar graph $G$ with $n_\ell$ odd either all of the limbs are near sun graphs or they are not. If they are, then we can remove one vertex from our ZFS on one leaf of an odd-leafed limb which ensures that $Z(G) \leq \frac{n}{2}$. If there is some limb that is not a sun graph, then there will be some extra vertex on the limb that does not need to be colored which guarantees that $n_\ell < \frac{n_F - \frac{n_B}{2}}{2}$ hence $Z(G) \leq \frac{n}{2}$. \qed

\section{Sharpness of Results}

We consider a family of graphs which we call \emph{sun graphs}. The sun graph $S_n$ is a graph on $2n$ vertices constructed by taking a cycle with vertices $v_1, \ldots, v_k$ and add vertices $u_1, \ldots u_n$ where by $u_i$ is adjacent to $v_i$ and $v_{i+1}$ (or $v_k$ and $v_1$ in the case of $u_1$). See Figure \ref{fig:sun} for $S_7$ and $S_8$.
We note that sun graphs are 2-connected outerplanar graphs for which $n_\ell$ is equal to $\frac{n}{2}$.  \begin{proposition}
    For any sun graph $S_k$ on $n = 2k$ vertices, we have
    $Z(S_k) = n_\ell = k = \frac{n}{2}$.
\end{proposition}  

\begin{proof}
By Theorems \ref{thm:lower_bound} and \ref{thm:ubound} we have $Z(S_k) \ge n_\ell$ and  $Z(S_k) \le n_\ell$ respectively.
\end{proof}

This shows that equality can be achieved in both Theorem \ref{thm:lower_bound} and Theorem \ref{thm:ubound}, and hence these theorems cannot be improved for general 2-connected outerplanar graphs.

\begin{figure}[H]
    \centering
    \begin{minipage}{.45\textwidth}
       \centering
        \ctikzfig{drawings/sunZFS}
        \subcaption{ZFS for the sun graph $S_8$ with 8 leaf cycles.} 
    \end{minipage}%
    \begin{minipage}{0.1\textwidth}
        \hspace{0mm}
    \end{minipage}%
    \begin{minipage}{.45\textwidth}
        \centering
        \ctikzfig{drawings/sun7ZFS}
        \subcaption{ZFS for the sun graph $S_7$ with 7 leaf cycles.}
    \end{minipage}
    \caption{}
    \label{fig:sun}
\end{figure}


\section{Conclusion} \label{sec:conclusion}

Theorem \ref{thm:ubound} demonstrates that 2-connected outerplanar graphs on $n$ vertices have $Z(G) \le \frac{n}{2}$. A natural question remains regarding how this result would generalize to planar graphs. It is worth noting that without any restriction, a planar graph can have $Z(G)$ close to $n$; take, for instance, the star graph on $n$ vertices. If allowing 2-connected planar graphs, the complete bipartite graph $K_{2,m}$ is planar has $Z(G) = m = n-2$. Hence, there are several ways one could consider generalizing our results to planar graphs with various restrictions. First is to consider maximal planar graphs, generalizing the work in \cite{hernandez2019zero} to planar graphs. Another, is to consider 3-connected planar graphs (in contrast to 2-connected outerplanar graphs). Finally, one could consider planar 3-trees (in contrast to maximal 2-connected outerplanar graphs which are 2-trees).

We do not attempt to answer any of these generalizations. However, we compute values for $Z(G)$ for smaller relevant planar graphs. The results of these computations are not sufficient to make a concrete conjecture, but do support the idea that $Z(G)$ is much less than $n$ for these restricted classes of planar graphs. The values in Tables \ref{tab.maximalplanar} and \ref{tab.planar3connected} indicate that for different classes of planar graphs, the value of $Z(G)$ is far less than the number of vertices.  This computation considers graphs from the {\it Wolfram} database \cite{Mathematica}; for emphasis, it is by no means complete or exhaustive. Table \ref{tab.maximalplanar} computes the zero forcing number for maximal planar graphs (i.e., for $n \ge 3$, planar graphs with $3n-6$ edges) and Table \ref{tab.planar3connected} computes $Z(G)$ for 3-connected planar graphs. A na\"ive look at small cases for each of these questions would indicate that perhaps $Z(G) \le 1/2 n + k$ for some constant $k$; however, the larger graphs (from 9 to 22 vertices) suggest that the 1/2 can be improved, perhaps to 1/3.

Most concretely, we ask the following:

\begin{question}
Consider maximal planar graphs, 3-connected planar graphs or planar 3-trees. For all such graphs $G$, is there a constant $\alpha < 1$ and a constant $k$ such that $Z(G) \le \alpha n + k$?
\end{question}

\begin{table}[H]
\centering 

$\begin{array}{|c|c|c|}
\hline
 \text{Graph Name} & n & Z(G) \\ \hline
 \text{FritschGraph} & 9 & 5 \\
 \{\text{JohnsonSkeleton},17\} & 10 & 5 \\
 \{\text{SierpinskiTetrahedron},2\} & 10 & 6 \\
 \text{GoldnerHararyGraph} & 11 & 5 \\
 \text{IcosahedralGraph} & 12 & 6 \\
 \text{SmallTriakisOctahedralGraph} & 14 & 6 \\
 \text{TetrakisHexahedralGraph} & 14 & 6 \\
 \text{PoussinGraph} & 15 & 6 \\
 \{\text{Apollonian},3\} & 16 & 7 \\
 \text{ErreraGraph} & 17 & 7 \\ \hline
\end{array}
$

\caption{Select maximal planar graphs on 9 to 22 vertices from the {\it Wolfram} database and their zero forcing numbers.}

\label{tab.maximalplanar}


\end{table}

\begin{table}[H]
\centering 

$\begin{array}{|c|c|c|}
\hline
 \text{Graph Name} & n & Z(G) \\ \hline
 \text{FritschGraph} & 9 & 5 \\
 \text{$\{$JohnsonSkeleton, 10$\}$} & 9 & 5 \\
 \text{$\{$King, $\{$3, 3$\}\}$} & 9 & 5 \\
 \text{$\{$Quartic, $\{$9, 5$\}\}$} & 9 & 5 \\
 \text{$\{$JohnsonSkeleton, 17$\}$} & 10 & 5 \\
 \text{$\{$JohnsonSkeleton, 86$\}$} & 10 & 5 \\
 \text{GoldnerHararyGraph} & 11 & 5 \\
 \text{HerschelGraph} & 11 & 5 \\
 \text{$\{$JohnsonSkeleton, 9$\}$} & 11 & 5 \\
 \text{$\{$JohnsonSkeleton, 11$\}$} & 11 & 5 \\
 \text{$\{$JohnsonSkeleton, 87$\}$} & 11 & 5 \\
 \text{$\{$Quartic, $\{$11, 61$\}\}$} & 11 & 5 \\
 \text{$\{$Quartic, $\{$11, 64$\}\}$} & 11 & 5 \\
 \text{$\{$Quartic, $\{$11, 120$\}\}$} & 11 & 5 \\
 \text{CuboctahedralGraph} & 12 & 6 \\
 \text{IcosahedralGraph} & 12 & 6 \\
 \text{$\{$JohnsonSkeleton, 54$\}$} & 13 & 4 \\
 \text{RhombicDodecahedralGraph} & 14 & 6 \\
 \text{SmallTriakisOctahedralGraph} & 14 & 6 \\
 \text{TetrakisHexahedralGraph} & 14 & 6 \\
 \text{$\{$JohnsonSkeleton, 89$\}$} & 14 & 6 \\
 \text{$\{$JohnsonSkeleton, 91$\}$} & 14 & 6 \\
 \text{PoussinGraph} & 15 & 6 \\
 \text{$\{$JohnsonSkeleton, 22$\}$} & 15 & 6 \\
 \text{$\{$JohnsonSkeleton, 65$\}$} & 15 & 6 \\
 \text{$\{$Apollonian, 3$\}$} & 16 & 7 \\
 \text{ErreraGraph} & 17 & 7 \\
 \text{$\{$Cubic, $\{$18, 1$\}\}$} & 18 & 6 \\
 \text{$\{$JohnsonSkeleton, 35$\}$} & 18 & 6 \\
 \text{$\{$JohnsonSkeleton, 36$\}$} & 18 & 6 \\
 \text{$\{$JohnsonSkeleton, 44$\}$} & 18 & 6 \\
 \text{$\{$JohnsonSkeleton, 92$\}$} & 18 & 6 \\
 \text{$\{$Wheel, 19$\}$} & 19 & 3 \\
 \text{$\{$JohnsonSkeleton, 6$\}$} & 20 & 7 \\
 \text{$\{$JohnsonSkeleton, 58$\}$} & 21 & 6 \\
 \text{$\{$Cubic, $\{$22, 1$\}\}$} & 22 & 6 \\
 \text{$\{$JohnsonSkeleton, 59$\}$} & 22 & 6 \\
 \text{$\{$JohnsonSkeleton, 60$\}$} & 22 & 6 \\ \hline
\end{array}$

\caption{Select 3-connected planar graphs from the {\it Wolfram} database with their zero forcing numbers. For each vertex count, only the graphs with the highest value of $Z(G)$ computed are shown. There may be other graphs with the same vertex count but a higher value of $Z(G)$. }

\label{tab.planar3connected}

\end{table}


\bibliographystyle{plain}
\bibliography{references}

\end{document}